\documentclass[12pt]{amsart}

\usepackage{xcolor}
\usepackage[normalem]{ulem}
\usepackage[T1]{fontenc}

\usepackage[margin=2.5cm]{geometry}
\usepackage{amsthm}
\usepackage{amsmath}
\usepackage{hyperref}
\usepackage{dsfont} 

\theoremstyle{definition}
\newtheorem{definition}{Definition}
\newtheorem{theorem}[definition]{Theorem}
\newtheorem{lemma}[definition]{Lemma}

\newtheorem{remark}[definition]{Remark}

\newcommand{\p}{\mathfrak{p}}
\newcommand{\q}{\mathfrak{q}}

\newcommand{\T}{{\mathbb{T}^d}}
\newcommand{\Z}{\mathbb{Z}}
\newcommand{\Ra}{R_\alpha}
\newcommand{\Rma}{R_{-\alpha}}
\newcommand{\oS}{\overline{S}}
\newcommand{\charone}{\mathds{1}}

\newcommand{\valpha}{\vec{\alpha}}
\newcommand{\Prob}{\mathbb{P}}

\newcommand{\mcal}{\mathcal{I}_k}

\title[Simple random walks on higher dimensional tori...]{Simple random walks on higher dimensional tori are mixing and not uniquely ergodic}
\author{Klaudiusz Czudek}
\address{Klaudiusz Czudek, Institute of Applied Mathematics, Faculty of Physics and Applied Mathematics, Gda{\'n}sk University of Technology, ul. Gabriela Narutowicza 11/12, 80-223 Gda{\'n}sk, Poland}
\email{klaudiusz.czudek@gmail.com}



\begin{document}

\begin{abstract}
It has been shown that in one dimension the environment viewed by the particle process (EVP process) in quasi periodic random environment is uniquely ergodic and mixing under mild additional assumptions. Here we construct an analytic quasi periodic environment on higher dimensional torus such that the EVP process is not uniquely ergodic. The stationary measures in this counterexample are necessarily atomless. We show also that the EVP process is  mixing with respect to any ergodic stationary measure under some smoothness assumption.
\end{abstract}

\maketitle

\section{Introduction}
\subsection{The environment viewed by the particle process}
The environment viewed by the particle process (EVP process) is an important tool in the study of random walks in random environments (see e.g. \cite{BS02}, \cite{Zeitouni_04}, \cite{Zeitouni_06}, \cite{Kozlov}). In the present paper we are interested in the problem of unique ergodicity and mixing of the EVP process of a random walk in quasi periodic environment.

Let us fix $d\ge 1$, an angle $\alpha\in \T$, and two strictly positive continuous functions $\p, \q \in C(\T)$ such that $\p(z)+\q(z)=1$ for every $z\in \T$. The Markov chain with transition kernel
\begin{equation}
\label{E:2.1}
p(z, \cdot)  = \p(z) \delta_{\Ra z} + \q(z) \delta_{\Rma z},
\end{equation}
where $R_\alpha$ denotes the rotation of angle $\alpha$, will be referred to as the EVP process of a random walk in quasi periodic environment (or simply the EVP process; the definition of the EVP process for an arbitrary random environment can be found in \cite{Zeitouni_04}). The corresponding Markov operator is of the form
\begin{equation}
\label{E:Markov}
P\mu(A)= \int_{\Ra^{-1} A} \p(z) \mu(dz) + \int_{\Rma^{-1} A} \q(z) \mu(dz), \quad P:\mathcal{M}(\T)\rightarrow \mathcal{M}(\T),
\end{equation}
where $A$ is Borel measurable and $\mathcal{M}(\T)$ is the space of Borel probability measures. Its dual operator\footnote{Dual means that the relation $\int U\varphi d\mu = \int \varphi dP\mu$ holds for every bounded Borel function $\varphi$ and Borel measure $\mu$.} acting on the space of Borel functions $\mathcal{B}(\T)$ is
\begin{equation}
\label{E:dual}
U\varphi(z) = \p(z) \varphi (\Ra z) + \q(z) \varphi(\Ra^{-1} z).
\end{equation}

A Borel measure $\mu$ is called stationary if $P\mu=\mu$. The space of all stationary measures is a convex and compact subset of $\mathcal{M}(\T)$ in the weak-$\ast$ topology. The extreme points of that subset are called ergodic measures. An ergodic stationary measure $\mu$ is called mixing if for every $\varphi$, $\psi\in C(\T)$ we have
$$
\bigg| \int_\T U^n\varphi(z) \psi(z) \mu(dz) - \int_\T \varphi(z) \mu(dz) \int_\T \psi(z) \mu(dz) \bigg|\to 0 \quad \textrm{as $n\to \infty$.} 
$$
We say that the process \eqref{E:2.1} is uniquely ergodic if there exists exactly one stationary measure. The process is called stable if it is uniquely ergodic and $P^n\nu$ converges in the weak-$\ast$ topology to the unique stationary measure for every $\nu\in\mathcal{M}(\T)$. Obviously the unique stationary measure of a stable process is necessarily mixing but not vice versa.

Since $\T$ is compact and $\p, \q$ are continuous, the Krylov-Bogoliubov technique ensures the existence of a stationary measure for \eqref{E:2.1}. If $\alpha_1, \cdots, \alpha_d, 1$ are not rationally independent then obviously the Markov chain \eqref{E:2.1} possesses infinitely many stationary measures with disjoint topological supports. However the situation is not so clear when $\alpha_1, \cdots, \alpha_d, 1$ are rationally independent.
\subsection{Known results}
Before we present the known results let us recall that a vector $\vec{\alpha}=(\alpha_1, \cdots, \alpha_d)$ is called Diophantine if there exists $K>0$ and $\gamma>0$ such that
$$
\inf_{m\in \Z} | (\vec{\alpha}, \vec{n}) - m | \ge \frac{K}{|\vec n|^\gamma},
$$
where $\vec{n}=(n_1,\cdots, n_d)\in \Z^d$, $|\vec{n}|=\sum_{i=1}^d |n_i|$ and $(\cdot,\cdot)$ stands for the standard scalar product. A function $\p$ is called symmetric if
\begin{equation}
\label{E:symmetric}
\int_{\T} \log \frac{\p(z)}{\q(z)}dz=0,
\end{equation}
and asymmetric otherwise.

Sinai proved \cite{Sinai_99} that if $\vec{\alpha}$ is Diophantine and $\p$ is sufficiently smooth then the EVP process \eqref{E:2.1} is uniquely ergodic and stable. Later Kaloshin and Sinai \cite{Kaloshin_Sinai_00} proved that if $\p$ is asymmetric then the same assertion holds even without the assumption that $\vec{\alpha}$ is Diophantine. Both of these results hold in arbitrary dimension. In dimension one Conze and Guivarc'h \cite{Conze_Guivarch_00} showed that if $\log\frac{p(z)}{q(\Ra z)}$ is continuous of bounded variation, $\vec{\alpha}$ is irrational then the EVP process \eqref{E:2.1} is uniquely ergodic. More recently the author in collaboration with D. Dolgopyat \cite{Czudek_Dolgopyat_24} proved that (no matter what the dimension is) every uniquely ergodic EVP process \eqref{E:2.1} is necessarily stable and obtained rates of mixing and the Central Limit Theorem in the case of Diophantine random environment.

In dimension one Bremont \cite{Bremont_99} constructed a continuous symmetric $\p$ with infinitely many ergodic stationary measures. In higher dimension Chevallier \cite{Chevallier} showed that for every $\alpha$ there exists a continuous symmetric $\p$ giving rise to a non uniquely ergodic process. He proved also that if $\vec{\alpha}$ is assumed to be Diophantine then $\p$ can be made Lipschitz. It is important to note that in both \cite{Bremont_99} and \cite{Chevallier} the authors construct the EVP processes with atomic stationary measures.

\subsection{The main result}

It is natural to ask if in dimension $d\ge 2$ smoothness of a symmetric $\p$ implies the unique ergodicity and mixing. The result below says that the unique ergodicity is not implied even if $\p$ is analytic. Moreover, if $\p$ is sufficiently smooth then every stationary measure is mixing and atomless. This answers Question 2.8 in \cite{DFS_21} in dimension $d\ge 2$.
\newpage

\begin{theorem}
\label{T:1}
Let us consider the EVP process \eqref{E:2.1}. Then the following statements hold.
\begin{enumerate}
\item Fix $d\ge 1$ and let $\mathcal{G}$ denote the space of smooth symmetric functions $\p : \T \rightarrow (0,1)$ equipped with $C^\infty(\T)$-topology. For every $\alpha\in \T$ let $\mathcal{G}_0\subseteq \mathcal{G}$ be the set of $\p$'s such that the random walk \eqref{E:2.1} is uniquely ergodic and stable, i.e. $ P^n \nu \to \mu \quad \textrm{in weak-$\ast$ topology}$
for every $\nu\in \mathcal{M}(\T)$, where $\mu$ is the unique stationary distribution. Then $\mathcal{G}_0$ contains a dense $G_\delta$ set.
\item For every $d\ge 2$ there exist $\alpha = (\alpha_1,\cdots, \alpha_d)$ and $\p, \q\in C^\omega(\T)$ (symmetric) such that $\alpha_1, \cdots, \alpha_d, 1$ are rationally independent and the random walk in \eqref{E:2.1} possesses at least two different ergodic stationary measures. 
\item For every $d\ge 1$ there exists $s=s(d)<\infty$ such that if $\p, \q \in C^s(\T)$ then every stationary measure of \eqref{E:2.1} is atomless. Moreover, every ergodic stationary measure $\mu$ is mixing, i.e. for every $\varphi$, $\psi\in C(\T)$ we have
$$
\bigg| \int_\T U^n\varphi(z) \psi(z) \mu(dz) - \int_\T \varphi(z) \mu(dz) \int_\T \psi(z) \mu(dz) \bigg|\to 0 \quad \textrm{as $n\to \infty$.} 
$$
\end{enumerate}
\end{theorem}

\section{Quasi-invariance equation}

As observed in \cite{Conze_Guivarch_00} the study of stationary measures of \eqref{E:2.1} is related to the study of quasi-invariance equation below. Let $h$ be a Borel measurable function with with $h(z)>0$ for every $z\in \T$. We say that a measure $\mu$ satisfies the quasi-invariance equation
\begin{equation}
\label{E:2.2}
R_{-\alpha}^\ast \mu = h\mu
\end{equation}
if the relation
$$
\int_\T \varphi(z) R_{-\alpha}^\ast \mu (dz) = \int_\T \varphi(z) h(z) \mu(dz)
$$
is satisfied for every $\varphi\in C(\T)$. If $\mu$ satisfies \eqref{E:2.2} and $P$ is the operator \eqref{E:Markov} then $P\mu$ is necessarily absolutely continuous with respect to $\mu$ and therefore it makes sense to consider the Perron-Frobenius operator $P^\ast: L^\infty (\mu) \rightarrow L^\infty(\mu)$ of $P$. A simple calculation shows that
$$
P^\ast\varphi(z) = \q (R_\alpha z) h(z) \varphi(\Ra z) + \p (\Rma z) h^{-1}(\Rma (z)) \varphi(\Rma(z)).
$$
As the consequence, if $h(z)=\frac{\p(z)}{\q(\Ra(z))}$ and $\mu$ satisfies the quasi-invariance equation \eqref{E:2.2} then $\mu$ is necessarily a stationary measure for \eqref{E:2.1}. Indeed, it is easy to check that in that case $P^\ast \charone =\charone$. Let us note the converse is not true anymore (however, see Theorem 3.1 in \cite{Conze_Guivarch_00}). A measure satisfying \eqref{E:2.2} is called ergodic if $R_{\alpha}(A)=A$ implies $\mu(A)=0$ or $\mu(A)=1$.

Let $\mu$ be a measure satisfying \eqref{E:2.2}. Let $T: L^1(\mu) \rightarrow L^1(\mu)$ denote the operator $\varphi\longmapsto h\cdot (\varphi\circ \Ra)$. It is easy to check that $T$ is a contraction on $L^1(\mu)$ (i.e. $\|T\varphi\|_{L^1(\mu)}\le \|\varphi\|_{L^1(\mu)}$), therefore the Chacon-Ornstein theorem implies that
\begin{equation}
\label{E:2.3}
\frac{\overline{S}_n(\varphi, z)}{\overline{S}_n(\psi, z)}:=\frac{\varphi(z)+T\varphi(z)+\cdots+T^{n-1}\varphi(z)}{\psi(z)+T\psi(z)+\cdots+T^{n-1}\psi(z)}
\end{equation}
converges $\mu$ a.e., if only $\psi, \varphi\in L^1(\mu)$ and $\psi> 0$. Let us finally observe that if $n,m\ge 0$ then
\begin{equation}
\label{E:Tproperty}
T^{n+m} \varphi(z) = T^n\mathds{1}(z) T^m\varphi(R_{\alpha}^nz) \quad z\in \T.
\end{equation}

In the sequel we denote  $\oS_n(\varphi, z)= \varphi(z)+T\varphi(z)+\cdots +T^{n-1} \varphi(z)$ and $S_n\varphi=\varphi+\varphi\circ \Ra+\cdots+\varphi\circ \Ra^{n-1}$, $S_{-n}\varphi=\varphi\circ \Rma + \cdots+\varphi\circ \Rma^{-n}$, $n>0$, $S_0\varphi(z)=0$.

\begin{lemma}
\label{L:coboundary}
If $f \in C(\T)$, $h=\exp f$ and $\mu$ solves $\Rma^\ast \mu = h\mu$, then $\nu = \exp(g) \mu$ solves $\Rma^\ast \nu= \exp(f+g\circ \Ra - g) \nu$.
\end{lemma}
\begin{proof}
Immediate.
\end{proof}

\begin{lemma}
\label{L:atomless}
For every $d\ge 1$ there exists $s=s(d)$ such that if $h\in C^s(\T)$ then every solution of \eqref{E:2.2} is atomless.
\end{lemma}
\begin{proof}
A solution of \eqref{E:2.2} is atomless if and only if $\sum_{n\in\Z} \exp(S_n f(z_0))<\infty$ for some $z_0\in \T$. Corollary to Theorem 61 \cite{Moshchevitin} on p. 492 implies that if only $s$ is sufficiently large then $\liminf_{|j|\to \infty} |S_jf(z)| < \infty$ for every $z\in \T$ and, consequently, $\sum_{j=0}^{n-1} \exp( S_jf(z))$ diverges for every $z\in \T$ to $\infty$ as $|n|\to \infty$.
\end{proof}

\begin{lemma}
\label{L:uniform}
For each $d\ge 1$ there exists $s=s(d)$ such that if $h\in C^s(\T)$ and $\Rma^\ast \mu= h \mu$ has exactly one solution $\mu$ (up to rescaling) then $\frac{\oS_n(\varphi, z)}{\oS_n(\charone, z)}$ converges uniformly to $\int_\T \varphi d\mu$ as $n\to\infty$.
\end{lemma}
\begin{proof}
Let $f=\log h \in C^s(\T)$. Then $\oS_n(\charone, z) = \sum_{j=0}^{n-1} \exp( S_jf(z))$, $n\in \Z$. Corollary to Theorem 61 \cite{Moshchevitin} on p. 492 implies that if only $s$ is sufficiently large then $\liminf_{|j|\to \infty} |S_jf(z)| < \infty$ for every $z\in \T$ and, consequently, $\sum_{j=0}^{n-1} \exp( S_jf(z))$ diverges to $\infty$ as $|n|\to\infty$. Obviously $\oS_n(\charone, z)$ is an incresing sequence of continuous functions, therefore Dini's theorem implies the divergence is uniform over $z\in \T$, i.e.
$$
\textrm{ $\inf_{z\in \T} \sum_{j=0}^{n-1} \exp( S_jf(z))\to \infty$ and $\inf_{z\in \T} \sum_{j=0}^{-n+1} \exp( S_jf(z))\to \infty$ as $n\to\infty$. }
$$
The above implies
\begin{equation}
\label{E:uniform1}
\frac{T^{n}\charone(z)}{\oS_{n+1}(\charone, z)}
=\frac{\exp(S_n f(z))}{1+\exp f(z)+\cdots+\exp(S_{n}f(z))}
\end{equation}
$$
=\bigg[\exp(S_{-n}f(\Ra^{n} z)+\cdots+ \exp S_{-1} f(\Ra^{n} z) +1  \bigg]^{-1} \to 0
$$
and
\begin{equation}
\label{E:uniform2}
\frac{1}{\oS_n(\charone, z)}= \frac{1}{\sum_{j=0}^{n-1} \exp( S_jf(z))} \to 0 \quad \textrm{as $n\to \infty$, uniformly in $z\in \T$.}
\end{equation}

Assume now that $\mu$ is the unique ergodic solution of $\Rma^\ast\mu=h\mu$ and there exist $\varphi\in C(\T)$, two sequences of points $(z^+_n)$, $(z^-_n)\subseteq \mathbb{T}^d$ and a sequence numbers $n_1<n_2<\cdots$ with
\begin{equation}
\label{E:uniform3}
\bigg|\frac{\oS_{n_k}(\varphi, z_k^+) }{\oS_{n_k}(\charone, z_k^+) } - \frac{\oS_{n_k}(\varphi, z_k^-) }{\oS_{n_k}(\charone, z_k^-) }\bigg|>\delta \quad \textrm{and} \quad \frac{\oS_{n_k}(\varphi, z_k^-) }{\oS_{n_k}(\charone, z_k^-)}\to \int_\T \varphi d\mu
\end{equation}
for some $\delta>0$ (the last follows from the Chacon-Ornstein theorem and the fact that $T$ is necessarily conservative ergodic if $\mu$ is ergodic, see \cite{Petersen} p.123 in Chapter 3.7 and Theorem 8.4 therein). Then similarly as in the proof of Proposition 4.1.13 in \cite{Katok_Hasselblatt_95} or Proposition 4.2 in \cite{Conze_Guivarch_00} using the diagonal method and separability of $C(\T)$ we can show that $\oS_{n_k}(\psi, z_k^+)/\oS_{n_k}(\charone, z_k^+)$ converges for every continuous $\psi\in C(\T)$ along a subsequence of $(n_k)$. This defines a continuous functional on $C(\T)$ and therefore a certain measure $\nu$. Observe that $S_n(T\psi, z)=S_n(\psi, z)-\psi(z)+T^n\psi(z)$ for $\psi\in C(\T)$, thus by \eqref{E:uniform1}, \eqref{E:uniform2} $\nu$ is necessarily a solution of $\Rma^\ast\nu = h\mu$. By \eqref{E:uniform3} $\nu\not=\mu$, which is a contradiction with the uniqueness of solution.
\end{proof}

\section{Random walks in quasiperiodic environment}
Let $\p: \T \rightarrow (0,1)$ be a continuous function (symmetric), $\alpha$ a $d$-dimensional vector, $z\in \T$, and let us define a sequence $(p_j, q_j)_j$, $j\in\mathbb Z$, $p_j=\mathfrak{p}(z+j\alpha)$, $q_j=1-p_j$. This sequence is called a quasiperiodic environment. A random walk in quasiperiodic environment $(p_j, q_j)_j$ is a process $(\xi_n)$ taking values in $\Z$ such that $\Prob_z( \xi_0 = 0 ) =1$ and 
\begin{equation}\label{E:random_walk}
\Prob_z( \xi_{n+1} = k+1 | \xi_n=k ) = p_k, \  \Prob_z( \xi_{n+1} = k-1 | \xi_n=k ) = q_k, k\in \mathbb{Z}, n\ge 0.
\end{equation}
We assume that $\Prob_z$ is the canonical measure defined on $\Z^\mathbb{N}$ equipped with the product $\sigma$-algebra. Given a Borel measure $\nu$ on $\T$ let us define
\begin{equation}
\label{E:annealed}
\Prob_\nu (A\times B) = \int_A \Prob_z(B) \nu(dz),
\end{equation}
where $B\subseteq \Z^\mathbb{N}$, $A\subseteq \T$, both measurable. We refer to \eqref{E:random_walk} as the quenched law of $(\xi_n)$ while to \eqref{E:annealed} as the annealed law. Random walk in quasi periodic environment is a particular example of random walks in random environment that were vastly studied in literature (see e.g. \cite{Zeitouni_04}, \cite{Zeitouni_06}, \cite{BS02}, \cite{Kozlov}). The quasi periodic environment was studied e.g. in \cite{Sinai_99}, \cite{Bremont_09B}, \cite{Dolgopyat_Goldsheid_21}, \cite{Dolgopyat_Goldsheid_19}, \cite{DFS_21}.

It can be shown that if $\p$ is symmetric then $(\xi_n)$ is recurrent with respect to both the annealed and quenched law (i.e. $\{\xi_n=0\}$ occurs infinitely many times almost surely, see e.g. \cite{Zeitouni_04}, Theorem 2.1.2). Obviously the dual operator \eqref{E:dual} can be expressed as
\begin{equation}
\label{E:dual_xi}
U^n\varphi(z) = \sum_{k\in \Z} \Prob_z(\xi_n=k) \varphi( \Ra^k z) \quad n\ge 1, z\in \T.
\end{equation}

Since $\p$ is separated from $0$ and $1$, there exists $\delta>0$ such that $\delta\le p_j \le 1-\delta$ for every $j$, $z\in \T$. That together with the recurrence of $(\xi_n)$ imply the following property called the strong ratio limit property (see \cite{Kingman_Orey_64} and \cite{Orey_61}, the proof can be found also in \cite{Freedman_83} Chapter 2.6 on p. 64).

\begin{lemma}
\label{L:Orey}
Let $\p$ be continuous symmetric. Then for every $z\in \T$
$$
\lim_{n\to \infty} \frac{\Prob_z(\xi_{2n}=2a)}{\Prob_z(\xi_{2n}=2b)} = \frac{\rho_{2a}(z)}{\rho_{2b}(z)},
$$
where
\begin{equation}\label{E:inv_measure_formula}
\rho_{n}(z)= \frac{\mathfrak{p}(z)}{\q(\Ra^n z)} \prod_{j=1}^{n} \frac{\mathfrak{p}(\Ra^j z)}{\q(\Ra^j z)}, \quad
\rho_{-n}(z)= \frac{\q(z)}{\mathfrak{p}(\Ra^{-1}z)} \prod_{j=1}^{n} \frac{\q(\Ra^{-j} z)}{\mathfrak{p}(\Ra^{-j} z)}
\end{equation}
for every $n>0$, $\rho_0(z)=1$.
\end{lemma}

\begin{remark}
Observe that with the notation of the preceding subsection $\rho_n(z)=T^n h(z)$, $n\ge 0$, $h=\exp \frac{\p}{\q\circ \Ra}$.
\end{remark}

\section{Auxiliary lemma}
Note that the following lemma deals with one-dimensional situation (i.e. with the notation from the preceding sections $d=1$).

\begin{lemma}
\label{L:auxiliary}
Let $\alpha_1=\frac p q$, $f_1$ a trigonometric polynomial with $\hat{f}_1(0)=0$ and $\hat{f}_1(j)=0$ for $j\ge q$. Let $f_2(t)= f_1(t)-c\sin\big(2\pi q (t-t_0)\big)$, $c>0$.  Let $B$ be an open subset containing $\{t_0, t_0+\frac 1 q, \cdots, t_0+ \frac {q-1} q \}$.
For $r\in\mathbb{Z}$ let us denote $\alpha=\alpha(r)=\alpha_1+\frac{1}{qr}$ and let $T$ be the operator $T: C(
\mathbb{T})\rightarrow C(\mathbb{T})$, $T\psi(t)=\exp f_2(t)\cdot (\psi\circ \Ra (t))$. Then for every $\delta>0$ there exists $r_0$ such that 
$$
\frac{\sum_{j=0}^{r-1} T^{jq}\charone (t)\cdot \charone_B
\big( \Ra^{jq} t \big)\cdot \oS_{q}(\charone, R_\alpha^{jq} t) }{\oS_{qr}(\mathds{1}, t)} >1-\delta,
$$
for every $t\in \mathbb{T}$, if only $r$ was chosen larger than $r_0$, where $\oS_n (\psi, t) = \sum_{j=0}^{n-1} T^j \psi(t)$.
\end{lemma}
\begin{proof}
It is sufficient to deal with $t_0=0$.
Let $\alpha = \alpha(r) = \alpha_1+ \frac{1}{qr}$ for some $r$ large, $r\in \mathbb{Z}$. Let us recall the formula
$$
\sin(2\pi q t)+ \sin(2\pi q \Ra t)+ \cdots + \sin(2\pi q \Ra^{n-1} t)
= \frac{\sin(2\pi q \frac{n+1}{2}\alpha ) \sin(2\pi q(\frac{n}{2}\alpha+t))}{\sin(\pi q\alpha)}.
$$
Applying the sine of sum formula to the first sine in the numerator yields
\begin{multline*}
\frac{\sin(2 \pi q \frac{n}{2} \alpha ) \cos(2 \pi q \frac{n}{2} \alpha ) \sin(2 \pi q (\frac{n}{2} \alpha+t) )}{\sin(\pi q \alpha)} + \cos\bigg(2 \pi q \frac{n}{2} \alpha \bigg)\sin\bigg(2 \pi q \bigg(\frac{n}{2} \alpha + t\bigg) \bigg)    \\
= \frac{\cos ( \pi q n \alpha )}{\sin(\pi q \alpha)}\sin\bigg(2 \pi q \frac{n}{2} \alpha \bigg)\sin\bigg(2 \pi q \bigg(\frac{n}{2} \alpha+t\bigg) \bigg) + O_{\alpha \to \alpha_1, n\to \infty}(1)\\
= \mathcal{Z}_r \bigg( \cos(2 \pi q t)- \cos (2\pi q \Ra^n t) \bigg) + O_{r\to\infty, n\to \infty}(1),
\end{multline*}
where $\mathcal{Z}_r =\frac{\cos(\pi q \alpha)}{2 \sin (\pi q \alpha)}$. Since $\hat{f}_1(j)=0$ for $j\ge q$ we have $S_n f_1(t)=O_{r\to \infty, n\to \infty}(1)$ uniformly in $t\in \mathbb{T}$. Those two facts combined give
$$S_nf_2(t) = c\mathcal{Z}_r \bigg( \cos(2\pi q \Ra^n t) - \cos(2 \pi q t) \bigg) + O_{r\to \infty, n\to \infty}(1),$$
and therefore
$$
T^n\charone (t) = \exp S_n f_2(t) = \exp \bigg( c\mathcal{Z}_r \bigg( \cos(2\pi q \Ra^n t) - \cos(2 \pi q t) \bigg) + O_{r\to \infty, n\to \infty}(1) \bigg),
$$
where $O_{r\to \infty, n\to \infty}(1) \bigg)$ uniformly in $t\in \mathbb{T}$. 

We have $T^{n+m}\varphi(t)=T^n\mathds{1}(t)T^m\varphi(R_\alpha^n t)$ for $n, m\ge 0$, which implies that for $j\le qr-1$
$$
T^{jq}\charone (t)  \oS_q(\charone, R^{jq}_\alpha t)
=\sum_{l=0}^{q-1} T^{jq+l} \charone(t) 
=\sum_{l=0}^{q-1} \exp \big( S_{jq+l}\charone(t) \big)
$$
$$
= \sum_{l=0}^{q-1} \exp\bigg( \bigg( c\mathcal{Z}_r \bigg( \cos(2\pi q \Ra^{jq+l} t) - \cos(2 \pi q t) \bigg) + O_{r\to \infty, n\to \infty}(1) \bigg).
$$

Let $\eta>0$ be so small that the $\eta$ neighbourhood of $F=\{0, \frac{1}{q}, \cdots, \frac{q-1}{q} \}$ is contained in $B$ and $\textrm{dist}( \Ra^j t, F)>\eta$ for $t\not\in B$, $j=0,1\cdots, q-1$. Let $\tilde{B}$ be an open set containing $F$, $\tilde{B}\subseteq B$, and such that $\textrm{dist}( \Ra^j t, F)<\eta/2$ for $t\in \tilde{B}$, $j=0, 1, \cdots, q-1$. $\textrm{dist}( \Ra^j t, F)>\eta$ for $t\not\in B$, $j=0,1\cdots, q-1$. \footnote{Observe that the definition of $\tilde B$ depends on $\alpha$ (and thus $r$) but in fact if $\tilde B$ is good for some $r_0$ fixed, then it is also good for all $r \ge r_0$. That is crucial as in the end of the proof we pass to the limit $r\to \infty$.}

If $j$ is such that $\mathds{1}_{\tilde B}(\Ra^{jq} t)=1$  then
$$
T^{jq}\charone(t)  \oS_q(\charone, R^{jq}_\alpha t) \ge q \exp\bigg( c\mathcal{Z}_r \cos(\pi q \eta ) \bigg) \exp\bigg( -c\mathcal{Z}_r \cos(2\pi q t) + O_{r\to\infty, n\to \infty, \eta\to 0}(1) \bigg).
$$
If $j$ is such that $\mathds{1}_{ B}(\Ra^{jq} t)=0$ then 
$$
T^{jq}\charone (t) \oS_q(\charone, R^{jq}_\alpha t) 
\le q \exp\bigg( c\mathcal{Z}_r \cos(2 \pi q \eta) \bigg) \exp\bigg( -c\mathcal{Z}_r \cos(2\pi q t) + O_{r\to\infty, n\to \infty, \eta\to 0}(1) \bigg).
$$ 
The points $\Ra^{qj} t$, $j=0,1, \cdots, r-1$ are equidistributed, therefore for at least $r \eta/4$ among all $j$'s, $j=0,1\cdots, r-1$, we have $\mathds{1}_{ \tilde B}(\Ra^{jq} t)=1$, thus
\begin{multline*}
\frac{\sum_{j=0}^{r-1} T^{jq}\charone(t) \mathds{1}_{\tilde{B}}( \Ra^{jq} t) \oS_q(\charone, R^{jq}_\alpha t)}{\sum_{j=0}^{r-1} T^{jq}\charone (t) (1-\mathds{1}_{B}(\Ra^{jq} t) \oS_q(\charone, R^{jq}_\alpha t) }\\
\ge \frac{q r \eta/4 \exp\bigg( c\mathcal{Z}_r \cos(\pi q \eta) \bigg) \exp\bigg( -c\mathcal{Z}_r \cos(2\pi q t) + O_{r\to\infty, n\to \infty, \eta\to 0}(1) \bigg)}{q r \exp\bigg( c\mathcal{Z}_r \cos(2\pi q \eta) \bigg) \exp\bigg( -c\mathcal{Z}_r \cos(2\pi q t) + O_{r\to\infty, n\to \infty, \eta\to 0}(1) \bigg)}\\
=\eta/4 \exp \bigg( c\mathcal{Z}_r\big(\cos(\pi q \eta) - \cos(2\pi q \eta)\big) + O_{n\to\infty, r\to\infty, \eta\to 0} (1) \bigg).
\end{multline*}
Recall that $\eta$ is fixed and $\mathcal{Z}_r\to\infty$ as $r\to \infty$. As a consequence, the above expression tends to infinity as $r\to \infty$. Since
$$
\frac{\sum_{j=0}^{r-1} T^{jq}\charone(t) \mathds{1}_{\tilde{B}}( \Ra^{jq} t) \oS_q(\charone, R^{jq}_\alpha t)}
{\sum_{j=0}^{r-1} T^{jq}\charone (t)  \oS_q(\charone, R^{jq}_\alpha t) } 
+
\frac{\sum_{j=0}^{r-1} T^{jq}\charone (t) (1-\mathds{1}_{B}(\Ra^{jq} t) \oS_q(\charone, R^{jq}_\alpha t) }
{\sum_{j=0}^{r-1} T^{jq}\charone (t)  \oS_q(\charone, R^{jq}_\alpha t) }
=1,
$$
which implies the assertion.
\end{proof}

\section{The proof of Theorem \ref{T:1}}
\subsection{Proof of (1)}

Let $\mathcal{H}$ be the space of all $C^\infty$ functions $h: \T \rightarrow \mathbb{R}$ with $h>0$ and $\int_\T \log h(z)dz=0$. Let $ {\mathcal{H}_0}\subseteq \mathcal{H}$ be the set of those $h$ for which the equation $\Rma^\ast \mu = h\mu$ possesses exactly one solution. Following the lines of the proof of Theorem 2 \cite{Chevallier} one can show that $\mathcal{H}_0$ is a dense $G_\delta$ in $\mathcal{H}$ with $C^\infty$-topology (the proof in \cite{Chevallier} is for $\|\cdot\|_\infty$-topology but it is easy to modify). Let $\Xi : \mathcal{H} \rightarrow \mathcal{G}$ be the map defined by $\Xi(h) = \frac{h}{1+h}$. This map is invertible, both $\Xi$ and $\Xi^{-1}$ are continuous, which proves $\Xi(\mathcal{H}_0)$ is a dense $G_\delta$ in $\mathcal{G}$.

We are going to prove that $\Xi(\mathcal{H}_0)\subseteq \mathcal{G}_0$. Let $\p\in \Xi(\mathcal{H}_0)$ and let us assume, contrary to the claim, that there exist two ergodic stationary measures $\mu_1$, $\mu_2$ of \eqref{E:2.1}. Both of them are necessarily solutions of $\Rma^\ast \mu=h_2\mu$, where $h_2 ={\p}/({\q\circ \Ra})\in \mathcal{H}_0$.  Indeed, by Lemma \ref{L:atomless} $\mu_1$, $\mu_2$ are atomless, thus Theorem 3.1 \cite{Conze_Guivarch_00} and the fact that $\p$ is symmetric imply that both $\mu_1$, $\mu_2$ are necessarily two different solutions of  $\Rma^\ast \mu=h_2\mu$. Now let $h\in \mathcal{H}_0$ satisfy $\p = \frac{h}{1+h}$. Since $\log h - \log h_2= \log \frac{\p}{\q} - \log \frac{\p}{\q \circ \Ra}  =\log \q\circ\Ra - \log \q$ is a coboundary, Lemma \ref{L:coboundary} and the fact that $h\in\mathcal{H}_0$ imply the equation $\Rma^\ast \mu=h_2\mu$ has the unique solution, which is a contradiction. Thus $\p\in \mathcal{G}_0$, which completes the proof of the unique ergodicity. Now stability follows from \cite{Czudek_Dolgopyat_24}.

\subsection{Proof of (2)}
Let us fix $\alpha_1, \alpha_2, \cdots, \alpha_d$, $d>2$, such that $\alpha_1, \cdots, \alpha_d, 1$ is rationally independent.  Fix $\p=\p(x_1, x_2)\in C(\mathbb{T}^2)$. If $\mu$ is a stationary distribution of \eqref{E:2.1} on two-dimensional torus with $\p=\p(x_1, x_2)\in C(\mathbb{T}^2)$ and $(\alpha_1, \alpha_2)$, then $\mu \otimes \textrm{Leb}_{d-2}$ is a stationary measure of \eqref{E:2.1} with $\tilde{\p}(x_1,\cdots, x_d):=\p(x_1, x_2)$ and $(\alpha_1, \alpha_2, \cdots, \alpha_d)$. Therefore to show the second part of Theorem 1 it is sufficient to deal with $d=2$.

It is sufficient to construct $\alpha=(\alpha_1, \alpha_2) \in \mathbb{T}^2$ (such that $\alpha_1, \alpha_2, 1$ are rationally independent), $h\in C^\omega(\mathbb{T}^2)$ of the form $h=\exp f$, $f\in C^\omega(\mathbb{T}^2)$, a continuous function $\varphi$ and infinite sequences of points $(z^+_n)$, $(z^-_n)\subseteq \mathbb{T}^2$ and numbers $q_1<q_2<\cdots$ with
\begin{equation}
\label{E:4.1goal}
\bigg|\frac{\oS_{q_n}(\varphi, z_n^+) }{\oS_{q_n}(\charone, z_n^+) } - \frac{\oS_{q_n}(\varphi, z_n^-) }{\oS_{q_n}(\charone, z_n^-) }\bigg|>0.5
\end{equation}
for every $n$. Indeed, then by Lemma \ref{L:uniform} the equation $\Rma^{\ast}\mu = h\mu$ has at least two different solutions. Let us then consider \eqref{E:2.1} with $\p=\frac{\exp f}{1+\exp f}$ and $\q=\frac{1}{1+\exp f}$. Put $h_2 ={\p}/({\q\circ \Ra})$. Then $\log h - \log h_2= \log \frac{\p}{\q} - \log \frac{\p}{\q \circ \Ra}  =\log \q\circ\Ra - \log \q$ is a coboundary, therefore $\Rma^\ast \mu=h_2\mu$ has at least two solutions by Lemma \ref{L:coboundary}. As we observed in Section 2 every solution of $\Rma^\ast \mu=h_2\mu$ is necessarily a stationary measure of \eqref{E:2.1}, which easily gives that \eqref{E:2.1} is not uniquely ergodic.

In the remaining of this part of the proof it is useful to denote
$$
T_{h, \vec{\alpha}} \varphi(z)= h\cdot (\varphi\circ R_{\valpha}), \quad S_n(h, \valpha, \varphi, z)=\sum_{j=0}^{n-1} T_{h, \valpha}^j\varphi(z).
$$
Let $\delta_1>\delta_2>\cdots$ be any sequence of positive reals with $\sum_{n=1}^\infty \delta_n<0.1$ and $\prod_{n=1}^\infty (1-\delta_n)>0.9$. We are going to define a function $\varphi\in C(\mathbb{T})$ and, by induction,
\begin{itemize}
\item a sequence of trigonometric polynomials $(f_n)$,
\item an increasing sequence of positive integers $(q_n)$ such that $q_n$ divides $q_{n+1}$, $n\ge 0$,
\item two sequences of circles $\ell_n^-$, $\ell_n^+\subseteq \mathbb{T}^2$  along with some open subsets $A_n^-, A_n^+$ containing them, where the circles $\ell_n^-$, $\ell_n^+$ are of the form $(a_nt, t)$, $t\in\mathbb{T}$, $(a_nt+\frac{1}{2}, t)$, $t\in \mathbb{T}$, where $a_{n+1} = a_1+q_1+q_2+\cdots+q_n$ for $n\ge 1$, $a_1=2$.
\end{itemize}
Those objects should be defined in such a way that the following conditions are satisfied. If we define
\begin{equation}
\label{E:4C1}
\valpha_n = \bigg( \sum_{j=1}^n \frac{a_j}{q_{j}}, \sum_{j=1}^{n} \frac{1}{q_{j}} \bigg),
\end{equation}
then
\begin{equation}
\label{E:4C2}
\frac{\oS_{q_n}(h_n, \valpha_n, \varphi, z_n^+) }{\oS_{q_n}(h_n, \valpha_n, \charone, z_n^+) }>(1-\delta_1)\cdots (1-\delta_n) \quad \textrm{and} \quad \frac{\oS_{q_n}(h_n, \valpha_n, \varphi, z_n^-) }{\oS_{q_n}(h_n, \valpha_n, \charone, z_n^-) }< \delta_1+\cdots+\delta_n
\end{equation}
for $z_n^-\in A_n^-$ and $z_n^+\in A_n^+$, $n\ge 1$, where $h_n=\exp f_n$ and $\varphi$ is a certain fixed continuous function.

Let $A_0^-=[-\frac{1}{10}, \frac{1}{10} ]\times \mathbb{T}$, $A_0^+=[\frac{4}{10}, \frac{6}{10}]\times\mathbb{T}$. Let $\varphi$ be any continuous function such that $\varphi(z)=0$ for $z\in A_0^-$, $\varphi(z)=1$ for $z\in A_0^+$ and $\|\varphi\|_\infty\le 1$. Further, let $\ell_1^-$, $\ell_1^+$ be the circles defined in parametric form as
\begin{equation}
\label{E:param}
\textrm{$\ell_1^-: (a_1t, t)$, $t\in [0,1]$, and $\ell_1^+: \bigg(a_1t+\frac{1}{2}, t\bigg)$, $t\in [0,1]$,}
\end{equation}
where $a_1=2$. Let $\vec{\alpha}_0=(0, 0 )$, $q_0=1$. Finally, let $f(x_1, x_2)=-\sin(4 \pi x_2)$.

Both $\ell_1^-$, $\ell_1^+$ are obviously $R_{\vec{\alpha}_0}$-invariant with the rotation number 0. The sets $A_0^- \cap \ell_1^-$ and $A_0^+ \cap \ell_1^+$ are open subsets of $\ell_1^-$ and $\ell_1^+$ containing $\{0, \frac{1}{2}\}$ (the coordinates are expressed in parametrization \eqref{E:param} in each of those circles). The function $f_1$ restricted to $\ell_1^-$ and $\ell_1^+$ equals $- \sin(4 \pi t)$. By Lemma \ref{L:auxiliary} there exists $r_1$ such that for $\vec{\alpha}_1 = ( \frac{a_1}{q_1}, \frac{1}{q_1} )$, $q_1=r_1q_0$, we have
\begin{multline*}
\frac{\oS_{q_1}(h_1, \valpha_1, \varphi, z)}{\oS_{q_1}(h_1, \valpha_1, \mathds{1} , z)}
\ge
\frac{\oS_{q_1}(h_1, \valpha_1, \mathds{1}_{A_0^+}, z )}{\oS_{q_1}(h_1, \valpha_1, \mathds{1} , z)}\\
=
\frac{\sum_{j=0}^{r_1-1} T_{h_1, \valpha_1}^{jq_0}\charone(z)\cdot \charone_{A_0^+}( R_{\vec{\alpha}_1}^{jq_0} z )\cdot \oS_{q_0}(h_1, \valpha_1, \mathds{1}, R_{\vec{\alpha}_1}^{jq_0} z) }
{\oS_{q_1}(h_1, \valpha_1, \mathds{1} , z)}
>1-\delta_1 \quad \textrm{for $z\in \ell_1^+$}
\end{multline*}
thus
\begin{equation}
\label{E:4.1}
\frac{\oS_{q_1}(h_1, \valpha_1, \varphi, z)}{\oS_{q_1}(h_1, \valpha_1, \mathds{1} , z)} >1-\delta_1  \quad \textrm{for $z\in A_1^+$,}
\end{equation}
where $A_1^+$ is a sufficiently small strip around $\ell_1^+$. Applying Lemma \ref{L:auxiliary} to $\ell_1^-$ and possibly increasing $r_1$ we obtain
\begin{multline*}
\frac{\oS_{q_1}(h_1, \valpha_1, \varphi, z)}{\oS_{q_1}(h_1, \valpha_1, \mathds{1} , z)}
\le
\frac{\oS_{q_1}(h_1, \valpha_1, \mathds{1}-\mathds{1}_{A_0^-}, z )}{\oS_{q_1}(h_1, \valpha_1, \mathds{1} , z)}\\
=
1-\frac{\sum_{j=0}^{r_1-1} T_{h_1, \valpha_1}^{jq_0}\charone(z)\cdot \charone_{A_0^-}( R_{\vec{\alpha}_1}^{jq_0} z )\cdot \oS_{q_0}(h_1, \valpha_1, \mathds{1}, R_{\vec{\alpha}_1}^{jq_0} z) }{\oS_{q_1}(h_1, \valpha_1, \mathds{1} , z)}<\delta_1 \quad \textrm{for $z\in \ell_1^-$,}
\end{multline*}
therefore
\begin{equation}
\label{E:4.2}
\frac{\oS_{q_1}(h_1, \valpha_1, \varphi, z)}{\oS_{q_1}(h_1, \valpha_1, \mathds{1} , z)}< \delta_1 \quad \textrm{for $z\in A_1^-$,}
\end{equation}
where $A_1^-$ is a sufficiently small strip around $\ell_1^-$. The equations \eqref{E:4.1}, \eqref{E:4.2} prove \eqref{E:4C2} for $n=1$.

Let us assume that the required objects are defined for some $n\ge 1$. Let $\ell_{n+1}^-$, $\ell_{n+1}^+$ be the circles defined in parametric form as
\begin{equation}
\label{E:param2}
\textrm{$\ell_{n+1}^-: (a_{n+1}t, t)$, $t\in \mathbb{T}$, and  $\ell_{n+1}^+: \bigg(a_{n+1}t + \frac{1}{2}, t\bigg)$, $t\in \mathbb{T}$,}
\end{equation}
where $a_{n+1}=a_1+q_1+\cdots+q_n$. Both of those circles are $R_{\valpha_n}$-invariant or, equivalently, $\vec{\alpha}_n\in \ell_{n+1}^-$. Indeed, if the following equalities are read mod 1 then (as $q_{j+1}$ is a multiple of $q_j$ for each $j=0,1, \cdots n-1$)
$$
a_{n+1} \sum_{j=1}^n \frac{1}{q_j} 
= \bigg(a_1+\sum_{j=1}^n q_j \bigg) \sum_{j=1}^n \frac{1}{q_j}
= \frac{a_1}{q_1} + \sum_{j=2}^n \frac{a_1+q_1+\cdots+q_{j-1}}{q_j}
=  \sum_{j=1}^n \frac{a_j}{q_j}.
$$
Let $f_{n+1}(x_1, x_2)= f_{n}(x_1, x_2) - c_{n+1}\sin(2\pi q_n x_2)$, where the number $c_{n+1}$ is so small that \eqref{E:4C2} still holds if $h_n=\exp f_n$ is replaced by $h_{n+1}=\exp f_{n+1}$. We assume also that
\begin{equation}
\label{E:4.decay2}
c_{n+1} \in (0, \exp(-q_n)).
\end{equation}

The points of intersection $\ell_{n+1}^+\cap \ell_n^+$ (in parametrization \eqref{E:param2} of $\ell_{n+1}^+$) are the points $j/q_n$, $j=0,1,\cdots, q_n-1$, and $A_{n}^+\cap \ell_{n+1}^+$ is an open subset of $\ell_{n+1}^+$ containing those points. The function $f_{n+1}$ restricted to $\ell_{n+1}^+$ in parametrization \eqref{E:param2} is of the form $\widetilde{f}(t)-c_{n+1} \sin(2\pi q_n t)$, where $\widetilde{f}(t)$ is a trigonometric polynomial with $\widehat{\widetilde{f}}(j)=0$ for $j\ge q_{n}$. The application of Lemma \ref{L:auxiliary} gives $r_{n+1}$ such that if $\valpha_{n+1}=\valpha_n+( \frac{a_{n+1}}{q_{n+1}} , \frac{1}{q_{n+1}} )$, where $q_{n+1}=q_n r_{n+1}$, then
$$
\frac{\sum_{j=0}^{r_{n+1}-1} T_{h_{n+1}, \valpha_{n+1}}^{jq_n}\charone(z)\cdot \charone_{A_n^+}( R_{\vec{\alpha}_{n+1}}^{jq_n} z )\cdot \oS_{q_{n}}(h_{n+1}, \valpha_{n+1}, \mathds{1}, R_{\vec{\alpha}_{n+1}}^{jq_n} z) }{\oS_{q_{n}}(h_1, \valpha_{n+1}, \mathds{1} , z)}>1-\delta_{n+1} \quad \textrm{for $z\in \ell_{n+1}^+$.}
$$
Using induction assumption \eqref{E:4C2} (together with what we assumed about $c_{n+1}$) and then the above consequence of Lemma \ref{L:auxiliary} we can write
\begin{multline*}
\frac{\oS_{q_{n+1}}(h_{n+1}, \valpha_{n+1}, \varphi, z)}{\oS_{q_{n+1}}(h_{n+1}, \valpha_{n+1}, \mathds{1}, z)} 
=
\frac{\sum_{j=0}^{r_{n+1}-1} T_{h_{n+1}, \valpha_{n+1}}^{jq_n}\charone(z)\cdot  \oS_{q_{n}}(h_{n+1}, \valpha_{n+1}, \varphi, R_{\vec{\alpha}_{n+1}}^{jq_n} z) }
{\oS_{q_{n+1}}(h_{n+1}, \valpha_{n+1}, \mathds{1}, z)} \\
\ge 
\frac{\sum_{j=0}^{r_{n+1}-1} T_{h_{n+1}, \valpha_{n+1}}^{jq_n}\charone(z)\cdot \charone_{A_n^+}( R_{\vec{\alpha}_{n+1}}^{jq_n} z )\cdot \oS_{q_{n}}(h_{n+1}, \valpha_{n+1}, \varphi, R_{\vec{\alpha}_{n+1}}^{jq_n} z) }
{\oS_{q_{n+1}}(h_{n+1}, \valpha_{n+1}, \mathds{1}, z)}\\
\ge
\frac{\sum_{j=0}^{r_{n+1}-1} T_{h_{n+1}, \valpha_{n+1}}^{jq_n}\charone(z)\cdot \charone_{A_n^+}( R_{\vec{\alpha}_{n+1}}^{jq_n} z )\cdot \oS_{q_{n}}(h_{n+1}, \valpha_{n+1}, \mathds{1}, R_{\vec{\alpha}_{n+1}}^{jq_n} z) }
{\oS_{q_{n+1}}(h_{n+1}, \valpha_{n+1}, \mathds{1}, z)} (1-\delta_1)\cdots (1-\delta_n)\\
>(1-\delta_1)\cdots (1-\delta_{n+1}) \quad \textrm{for $z\in \ell_{n+1}^+$.}
\end{multline*}

The points of intersection $\ell_{n+1}^-\cap \ell_n^-$ (in parametrization \eqref{E:param2} of $\ell_{n+1}^-$) are again the points $j/q_n$, $j=0,1,\cdots, q_n-1$, and $A_{n}^-\cap \ell_{n+1}^-$ is an open subsets of $\ell_{n+1}^-$ containing those points. The function $f_{n+1}$ restricted to $\ell_{n+1}^-$ is of the form $\widetilde{f}(t)-c_{n+1} \sin(2\pi q_n t)$ in parametrization \eqref{E:param2}, where $\widetilde{f}(t)$ is a trigonometric polynomial with $\widehat{\widetilde{f}}(j)=0$ for $j\ge q_{n}$. Lemma \ref{L:auxiliary} says that we can correct the previous choice of $r_{n+1}$ in such a way that
$$
\frac{\sum_{j=0}^{r_{n+1}-1} T_{h_{n+1}, \valpha_{n+1}}^{jq_n}\charone(z)\cdot \charone_{A_n^-}( R_{\vec{\alpha}_{n+1}}^{jq_n} z )\cdot \oS_{q_{n}}(h_{n+1}, \valpha_{n+1}, \mathds{1}, R_{\vec{\alpha}_{n+1}}^{jq_n} z) }{\oS_{q_{n}}(h_1, \valpha_{n+1}, \mathds{1} , z)}>1-\delta_{n+1} \quad \textrm{for $z\in \ell_{n+1}^-$,}
$$
which easily gives
$$
\frac{\sum_{j=0}^{r_{n+1}-1} T_{h_{n+1}, \valpha_{n+1}}^{jq_n}\charone(z)\cdot (1-\charone_{A_n^-}( R_{\vec{\alpha}_{n+1}}^{jq_n} z ))\cdot \oS_{q_{n}}(h_{n+1}, \valpha_{n+1}, \mathds{1}, R_{\vec{\alpha}_{n+1}}^{jq_n} z) }{\oS_{q_{n}}(h_1, \valpha_{n+1}, \mathds{1} , z)}<\delta_{n+1} \quad \textrm{for $z\in \ell_{n+1}^-$.}
$$
Using the above fact and the induction assumption \eqref{E:4C2} we have
\begin{multline*}
\frac{\oS_{q_{n+1}}(h_{n+1}, \valpha_{n+1}, \varphi, z)}{\oS_{q_{n+1}}(h_{n+1}, \valpha_{n+1}, \mathds{1}, z)}\\
=
\frac{\sum_{j=0}^{r_{n+1}-1} T_{h_{n+1}, \valpha_{n+1}}^{jq_n}\charone(z)\cdot (1-\charone_{A_n^-}( R_{\vec{\alpha}_{n+1}}^{jq_n} z ))\cdot \oS_{q_{n}}(h_{n+1}, \valpha_{n+1}, \varphi, R_{\vec{\alpha}_{n+1}}^{jq_n} z) }{\oS_{q_{n}}(h_1, \valpha_{n+1}, \mathds{1} , z)}\\
+
\frac{\sum_{j=0}^{r_{n+1}-1} T_{h_{n+1}, \valpha_{n+1}}^{jq_n}\charone(z)\cdot \charone_{A_n^-}( R_{\vec{\alpha}_{n+1}}^{jq_n} z )\cdot \oS_{q_{n}}(h_{n+1}, \valpha_{n+1}, \varphi, R_{\vec{\alpha}_{n+1}}^{jq_n} z) }{\oS_{q_{n}}(h_1, \valpha_{n+1}, \mathds{1} , z)}\\
<
\delta_{n+1} + 
\frac{\sum_{j=0}^{r_{n+1}-1} T_{h_{n+1}, \valpha_{n+1}}^{jq_n}\charone(z)\cdot \charone_{A_n^-}( R_{\vec{\alpha}_{n+1}}^{jq_n} z )\cdot \oS_{q_{n}}(h_{n+1}, \valpha_{n+1}, \mathds{1}, R_{\vec{\alpha}_{n+1}}^{jq_n} z) }{\oS_{q_{n}}(h_1, \valpha_{n+1}, \mathds{1} , z)}(\delta_1+\cdots+\delta_n)\\
\le \delta_1+\cdots+\delta_{n+1} \quad \textrm{for $z\in \ell_{n+1}^-$.}
\end{multline*}
Let $A_{n+1}^+$ and $A_{n+1}^-$ be small strips around $\ell_{n+1}^+$ and $\ell_{n+1}^-$. Then 
$$
\frac{\oS_{q_{n+1}}(h_{n+1}, \valpha_{n+1}, \varphi, z_+)}{\oS_{q_{n+1}}(h_{n+1}, \valpha_{n+1}, \mathds{1}, z_+)}>(1-\delta_1)\cdots (1-\delta_{n+1}), \quad 
\frac{\oS_{q_{n+1}}(h_{n+1}, \valpha_{n+1}, \varphi, z_-)}{\oS_{q_{n+1}}(h_{n+1}, \valpha_{n+1}, \mathds{1}, z_-)}<\delta_1+\cdots+\delta_{n+1}
$$
for $z^+\in A_{n+1}^+$ and $z^-\in A_{n+1}^-$.

Recall that our assumption on $c_{n+1}$ says that the inequalities \eqref{E:4C2} hold with $h_n$ replaced by $h_{n+1}$. Since $r_{n+1}$ can be taken arbitrarily large, we can also assume that $\valpha_{n+1}$ is sufficiently close to $\valpha_n$ that the inequalities \eqref{E:4C2} obtained in the previous steps still hold with $h_{n}$ replaced by $h_{n+1}$ and $\valpha_{n}$ replaced by $\valpha_{n+1}$.

This completes the induction procedure. Let us take $\valpha=\lim_{n\to \infty} \valpha_n$ and $f=\sum_{n=1}^\infty f_n$. By \eqref{E:4.decay2} $f$ is analytic. By what have been said in the last paragraph of the induction procedure the inequalities \eqref{E:4C2} still hold with $\valpha_n$ replaced by $\valpha$ and $f_n$ replaced by $f$. This shows \eqref{E:4.1goal} by the choice of the sequence $(\delta_n)$.

\subsection{Proof of (3)}

Fix $d\ge 1$, $\p \in C^s(\T)$ symmetric, where $s=s(d)$ is given by Corollary to Theorem 61 \cite{Moshchevitin}. Denote $f(x)=\log\frac{\p(x)}{\q(x)}$, $h(x)=\exp f(x)$, and recall that we assume $\int_{\T} f(x)dx=0$. Let $\mu$ be any ergodic stationary measure of \eqref{E:2.1}. Then, as explained in Section 2, $\mu$ is a solution of the quasi-invariance equation. Let $T$ be the corresponding operator on $L^1(\mu)$, $T\varphi = h\cdot ( \varphi\circ \Ra)$, and observe that $T^2: L^1(\mu) \rightarrow L^1(\mu)$ is conservative ergodic (see \cite{Petersen} just before Lemma 7.2), thus the Chacon-Ornstein theorem (Theorem 8.4 \cite{Petersen}) applied to $T^2$ implies that for any fixed $\varphi \in C(\T)$
$$
\frac{\varphi(z)+T^2\varphi(z)+\cdots+T^{2n}\varphi(z)}{\mathds{1}(z)+T^2\mathds{1}(z)+\cdots+T^{2n}\mathds{1}(z)}\to \int_{\T} \varphi d\mu 
$$
for $\mu$ a.e $z\in \T$. Without loss of generality we can assume that $\int_{\T} \varphi d\mu=0$, therefore there exists $r$ such that
\begin{equation}
\label{E:Orey}
\mu\bigg( \bigg\{ z \in \T: \bigg| \frac{\varphi(z)+T^2\varphi(z)+\cdots+T^{2r}\varphi(z)}{\mathds{1}(z)+T^2\mathds{1}(z)+\cdots+T^{2r}\mathds{1}(z)} \bigg|>\varepsilon/2 \bigg\} \bigg) < \varepsilon/2
\end{equation}
Let $(\xi_n)$ be the random walk in random environment with $\Prob_z$, $\Prob_\mu$ defined in Section 3. Given $n_0$, let $G$ be the set of all $z$ such that
\begin{equation}
\label{E:defG}
\forall_{n\ge n_0} \bigg| \sum_{j=0}^r \frac{\mathbb{P}_z(\xi_{2n} =2j)}{\mathbb{P}_z(\xi_{2n}\in [0, 2r])}\varphi \big( R^{2j}_\alpha (z) \big)  \bigg| < \varepsilon
\end{equation}
By \eqref{E:Orey}, Lemma \ref{L:Orey} and the following Remark if $n_0$ is sufficiently large then $\mu(G)>1-\varepsilon$.
Define 
$$
\mathcal{I}_k := [2kr, 2(k+1)r), \quad k\in \mathbb{Z}
$$
and $\tau_k = \min \{ n \ge 1: \textrm{ $\Ra^{\xi_{2n}}z \in G$ and $ \xi_{2n} \in \mcal$ } \}$.
\begin{lemma}
\label{L:probability}
If $n$ is sufficiently large, then
$$
\sum_{k\in \mathbb{Z}} \sum_{l\in \mathcal{I}_k} \int_{\T} \Prob_z(\xi_{2n}=l , \tau_k > 2n-2n_0)\mu(dz) < 2\varepsilon.
$$
\end{lemma}
We postpone the proof until the next subsection. Now we have by \eqref{E:dual_xi}
\begin{multline*}
U^{2n} \varphi(z)
=\sum_{k\in \Z} \sum_{l\in \mathcal{I}_k} \Prob_z(\xi_{2n} = l) \varphi\big( \Ra^l z \big)\\
= \sum_{k\in \mathbb{Z}} \sum_{l\in \mathcal{I}_k} \sum_{m=1}^{n-n_0} \sum_{l'\in \mathcal{I}_k} \Prob_z \big( \xi_{2n}=l, \tau_k = 2m, \xi_{2m}=l' \big)\varphi\big( \Ra^l z \big)\\
+ 
\sum_{k\in \mathbb{Z}} \sum_{l\in \mathcal{I}_k} \Prob_z \big( \xi_{2n}=l, \tau_k \ge 2n-2n_0\big) \varphi\big( \Ra^l z \big)\\
=\sum_{k\in \mathbb{Z}} \sum_{m=1}^{n-n_0}\sum_{l'\in \mcal} \Prob_z \big( \xi_{2n}\in \mcal, \tau_k = 2m, \xi_{2m}=l' \big)
\cdot
 \sum_{l\in \mcal} \frac{\Prob_z(\xi_{2n}=l | \tau_k=2m, \xi_{2m}=l')}{\Prob_z(\xi_{2n}\in \mcal | \tau_k=2m, \xi_{2m}=l')} \varphi(R^l_\alpha z)\\
 +\sum_{k\in \mathbb{Z}} \sum_{l\in \mathcal{I}_k} \Prob_z(\xi_{2n}=l , \tau_k \ge 2n-2n_0) \varphi(\Ra^l z).
\end{multline*}
By the definition of $\tau_k$ and $G$ we have
\begin{equation}
\label{E:4.2.1}
\bigg|\sum_{l\in \mcal} \frac{\Prob_z(\xi_{2n}=l | \xi_{2m}=l', \Ra^{l'} z \in G)}{\Prob_z(\xi_{2n}\in \mcal | \xi_{2m}=l', \Ra^{l'} z \in G)} \varphi(R^l_\alpha z) \bigg| < \varepsilon
\end{equation}
Therefore using Lemma \ref{L:probability} and \eqref{E:4.2.1}
\begin{multline*}
\bigg|\int_{\T} U^{2n} \varphi(z) \psi(z) \mu(dz) \bigg|
\le 
\| \psi \|_\infty \int_{\T} \big| U^{2n} \varphi(z) \big| \mu(dz)\\
\le 
\varepsilon \| \psi \|_\infty  
+
\|\psi\|_\infty \|\varphi\|_\infty \sum_{k\in \mathbb{Z}} \sum_{l\in \mathcal{I}_k} \int_{\T} \Prob_z(\xi_{2n}=l , \tau_k > 2n-2n_0) \mu(dz)\\
\le
 \varepsilon\|\psi\|_\infty + 2\varepsilon \| \psi \|_\infty \|\varphi\|_\infty
\end{multline*}
for $n$ sufficiently large. Since $\varepsilon$ can be arbitrarily small, this proves the assertion for $U^{2}$. To complete the proof it is sufficient to observe that $U^{2n+1}\varphi= U^{2n}(U\varphi)$ and $U\varphi\in C(\T)$.

\subsection{The proof of Lemma \ref{L:probability}}

Let us recall that $\Prob_\mu$ was defined in \eqref{E:annealed}. Given $n$ let $\kappa(\omega)\in \Z$ be the index with $\xi_{2n}(\omega)\in \mathcal{I}_{\kappa(\omega)}$. First we claim that for every $r$ and $n_0$
\begin{equation}
\label{E:4.3.1}
\lim_{n\to \infty} \Prob_\mu \bigg( \bigg\{ \textrm{there exists $l\in \mathcal{I}_{\kappa(\omega)}$ such that $l\not\in \{\xi_0,\cdots, \xi_{2n-2n_0}\}$}\bigg\} \bigg) = 0
\end{equation}
If, contrary to the claim, the limit inferior of the above expression is positive, then
\begin{equation}
\label{E:4.3.2}
\liminf_{n\to\infty} \Prob_\mu (\max_{j\le n-1} \xi_j<\xi_n) >0.
\end{equation}
or
\begin{equation}
\label{E:4.3.3}
\liminf_{n\to\infty} \Prob_\mu (\min_{j\le n-1} \xi_j>\xi_n) >0.
\end{equation}
Indeed, for each $n$
\begin{multline*}
\Prob_\mu \bigg( \bigg\{ \textrm{there exists $l\in \mathcal{I}_{\kappa(\omega)}$ such that $l\not\in \{\xi_0,\cdots, \xi_{2n-2n_0}\}$}\bigg\} \bigg)\\
=
\Prob_\mu \bigg(  \bigg\{ \textrm{there exists $l\in \mathcal{I}_{\kappa(\omega)}$ such that $l\not\in \{\xi_0,\cdots, \xi_{2n-2n_0}\}$}\bigg\}
\cap \{\xi_{2n}\ge 0 \} \bigg)\\
+
\Prob_\mu \bigg(\bigg\{ \textrm{there exists $l\in \mathcal{I}_{\kappa(\omega)}$ such that $l\not\in \{\xi_0,\cdots, \xi_{2n-2n_0}\}$}\bigg\}
\cap \{\xi_{2n}< 0 \} \bigg),
\end{multline*}
therefore the limit inferior of at least one of the summands on the right-hand side is positive. Let us assume the first one is positive. The goal is to show that \eqref{E:4.3.2} holds in that case (it is analougous to show that \eqref{E:4.3.3} holds if the limit inferior of the second summand is positive). By the definition of $\mathcal{I}_k$ and $\kappa$
\begin{multline*}
\bigg\{ \textrm{there exists $l\in \mathcal{I}_{\kappa(\omega)}$ such that $l\not\in \{\xi_0,\cdots, \xi_{2n-2n_0}\}$}\bigg\} \cap \{ \xi_{2n}>0 \}\\
\subseteq
\bigg\{ 2\kappa(\omega)r\le \xi_{2n} \bigg\}
\cap
\bigg\{ \max ( \xi_0,\cdots, \xi_{2n-2n_0} ) <2(\kappa(\omega)+1)r-1 \bigg\}
\cap \{ \xi_{2n}>0 \}\\
\subseteq
\{ \max(\xi_0, \cdots, \xi_{2n-2n_0}) - \xi_{2n-2n_0} < 2r+n_0 \},
\end{multline*}
thus $\liminf \Prob_\mu \big( \max(\xi_0, \cdots, \xi_{2n-2n_0}) - \xi_{2n-2n_0} < 2r+n_0 \big)>0$. On the other hand,
\begin{multline*}
\{ \max(\xi_0, \cdots, \xi_{2n-2n_0}) - \xi_{2n-2n_0} <2r + n_0 \} \cap \bigcap_{j=0}^{2r-n_0} \{ \xi_{2n-2n_0+j} - \xi_{2n-2n_0+j-1}=1\}\\
\subseteq
\bigg\{ \max_{j<2n+2r-n_0} \xi_j < \xi_{2n+2r-n_0} \bigg\}.
\end{multline*}
Since $\p>0$ we have
$$
\Prob_\mu \bigg( \bigcap_{j=0}^{2r-n_0} \{ \xi_{2n-2n_0+j} - \xi_{2n-2n_0+j-1}=1\} \bigg)= \int_\T \prod_{j=0}^{2r-n_0} \p(R^{\xi_{2n-2n_0+j-1}}_\alpha z ) \mu(dz)=:\eta > 0.
$$
The Markov property and the above give
$$
\Prob_\mu \bigg( \max(\xi_0, \cdots, \xi_{2n-2n_0}) - \xi_{2n-2n_0} <2r + n_0 \bigg)\le\frac{1}{\eta}\Prob_\mu \bigg( \max_{j<2n+2r-n_0} \xi_j < \xi_{2n+2r-n_0} \bigg),
$$
which yields
\begin{multline*}
0<\liminf_{n\to \infty} \Prob_\mu\bigg( \bigg\{ \textrm{there exists $l\in \mathcal{I}_{\kappa(\omega)}$ such that $l\not\in \{\xi_0,\cdots, \xi_{2n-2n_0}\}$}\bigg\} \cap \{ \xi_{2n}>0 \} \bigg)\\
\le
   \liminf_{n\to \infty} \frac{1}{\eta}\Prob_\mu\bigg( \max_{j<n} \xi_j < \xi_n \bigg),
\end{multline*}
which proves \eqref{E:4.3.2}. We know therefore  that \eqref{E:4.3.2} or \eqref{E:4.3.3} holds and, without loss of generality, we assume the first of those is valid (the proof in the second case is analogous).

Let $\mathcal{X}=\mathbb{T}^\mathbb{Z}$, $\mathcal{F}^\infty$ be the product $\sigma$-algebra on $\mathcal{X}$, $\Prob_\mu^\infty$ be the Markov measures on $(\mathcal{X}, \mathcal{F}^\infty)$ defined on the cylinder $A_0\times\cdots \times A_n$, $A_i\in\mathcal{B}(\mathbb{T})$ by
$$
\Prob_\mu^\infty (A_0\times\cdots \times A_n)= \int_{A_0} \int_{A_1} \cdots \int_{A_{n-1}} p(x_{n-1}, A_n) p(x_{n-2}, dx_{n-1})\cdots p(x_1, dx_2)p(x_0, dx_1) \mu(dx_0),
$$
where $p(\cdot, \cdot)$ is defined by \eqref{E:2.1}. Let $\sigma$ denote the left shift on $\mathcal{X}$. Clearly $(\mathcal{X}, \mathcal{F}^\infty, \Prob_\mu^\infty, \sigma)$ is an ergodic invertible measure preserving dynamical system (recall we assumed $\mu$ to be ergodic). Let
$$\rho(\overline{x})=\left\{
\begin{array}{ccc}
1&\mbox{}&\textrm{if $x_1=x_0+\alpha$}\\
-1&\mbox{}&\textrm{if $x_1=x_0-\alpha$}\\
0&\mbox{}&\textrm{otherwise,}
\end{array}
\right.$$
where $\overline{x}=(\cdots, x_{-1}, x_0, x_1, \cdots )\in \mathcal{X}$. Since $\mu$ is stationary, it satisfies $R_{-\alpha}^\ast \mu = h\mu$ for $h=\frac{\p}{\q \circ \Ra } \mu $, thus we have
\begin{multline*}
 \int_\mathcal{X} \rho(\overline{x}) \Prob_\mu^\infty(d\overline{x})
 = \int_\T (\p(z)-\q(z))\mu(dz)
 = \int_\T \q\circ \Ra (z) \frac{\p(z)}{\q\circ \Ra(z)} \mu(dz)-\int_\T \q(z)\mu(dz)\\
 = \int_\T \q\circ \Ra (z) R^\ast_{-\alpha} \mu(dz) -\int_\T \q(z)\mu(dz)
 = \int_\T \q(z) \mu(dz) - \int_\T \q(z) \mu(dz)=0
\end{multline*}
and
$$
\Prob_\mu \bigg(\max_{1\le j\le n-1} \xi_j<\xi_n \bigg) 
= \Prob_\mu^\infty \bigg( \max_{0\le j\le n-1} S_j\rho(\overline{x}) - S_n\rho(\overline{x})<0 \bigg)
= \Prob_\mu^\infty \bigg( \max_{1\le j\le n} - S_{-j} \rho(\sigma^n \overline{x})<0 \bigg)
$$
$$
=\Prob_\mu^\infty \bigg( \min_{1\le j\le n} S_{-j} \rho( \overline{x})>0 \bigg).
$$
By Corollary 1.14 \cite{Petersen} applied to $(\mathcal{X}, \mathcal{F}^\infty, \Prob_\mu^\infty, \sigma^{-1} )$ and $-\rho$ this converges to zero as $n\to \infty$, which is a contradiction with \eqref{E:4.3.2}.

Now we are ready to complete the proof of Lemma \ref{L:probability}. We have for $k\in\mathbb{Z}$ fixed
$$
\bigg\{  \textrm{$\mathcal{I}_{k}\subseteq \{\xi_0,\cdots, \xi_{2n-2n_0}\}$ }  \bigg\} \cap \{\tau_k>2n-2n_0 \} \subseteq \bigcap_{l\in \mathcal{I}_k} \{R^l_\alpha z\not\in G\},
$$
therefore
\begin{multline*}
\sum_{k \in \mathbb{Z}} \Prob_\mu(\xi_{2n}\in \mathcal{I}_k, \tau_k>2n-2n_0)\\
\le
\sum_{k \in \mathbb{Z}} \Prob_\mu \bigg(  \bigg\{ \textrm{there exists $l\in \mathcal{I}_{k}$ such that $l\not\in \{\xi_0,\cdots, \xi_{2n-2n_0}\}$}\bigg\}  \bigg)
+
\sum_{k \in \mathbb{Z}} \sum_{l\in \Z} \Prob_\mu ( R^l_\alpha z\not\in G )\\
=
\Prob_\mu \bigg(\bigg\{ \textrm{there exists $l\in \mathcal{I}_{\kappa(\omega)}$ such that $l\not\in \{\xi_0,\cdots, \xi_{2n-2n_0}\}$}\bigg\}  \bigg)
+\Prob_\mu( R^{\xi_{2n}}_\alpha z\not\in G).
\end{multline*}
The first summand is less than $\varepsilon$ for $n$ sufficiently large by \eqref{E:4.3.1}. The second one can be rewritten using the operator \eqref{E:Markov} as 
$$
\int_\T U^{2n} \mathds{1}_{\T\setminus G} (z) \mu(dz)
= \int_\T \mathds{1}_{\T\setminus G} (z) P^{2n}\mu(dz)
= \int_\T \mathds{1}_{\T\setminus G} (z) \mu(dz) = \mu(\T \setminus G)< \varepsilon
$$
by the choice of $n_0$. By \eqref{E:annealed} this completes the proof of Lemma \ref{L:probability}.

\bibliographystyle{plain}
\bibliography{Bibliography}

\begin{thebibliography}{10}

\bibitem{BS02}
E.~Bolthausen and A.-S. Sznitman.
\newblock On the static and dynamic points of view for certain random walks in
  random environment.
\newblock {\em Methods Appl. Anal.}, 9(3):345--375, 2002.

\bibitem{Bremont_99}
J.~Bremont.
\newblock Comportement des sommes ergodiques pour des rotations et des
  fonctions continues peu r\'{e}gulieres, 1999.

\bibitem{Bremont_09B}
J.~Bremont.
\newblock Random walk in quasi-periodic random environment.
\newblock {\em Stoch. Dyn.}, 9(1):47--70, 2009.

\bibitem{Chevallier}
N.~Chevallier.
\newblock Mesures quasi-invariantes sur le tore {${\Bbb T}^d$}.
\newblock {\em J. Anal. Math.}, 92:371--383, 2004.

\bibitem{Conze_Guivarch_00}
J.-P. Conze and Y.~Guivarc'h.
\newblock Marches en milieu al\'{e}atoire et mesures quasi-invariants pour un
  syst\`eme dynamique.
\newblock volume 84/85, pages 457--480. 2000.
\newblock Dedicated to the memory of Anzelm Iwanik.

\bibitem{Czudek_Dolgopyat_24}
K.~Czudek and D.~Dolgopyat.
\newblock The central limit theorem and rate of mixing for simple random walks
  on the circle.
\newblock {\em ALEA Lat. Am. J. Probab. Math. Stat.}, 21(2):1853--, 2024.

\bibitem{DFS_21}
D.~Dolgopyat, B.~Fayad, and M.~Saprykina.
\newblock Erratic behavior for 1-dimensional random walks in a {L}iouville
  quasi-periodic environment.
\newblock {\em Electron. J. Probab.}, 26:Paper No. 66, 36, 2021.

\bibitem{Dolgopyat_Goldsheid_19}
D.~Dolgopyat and I.~Goldsheid.
\newblock Invariant measure for random walks on ergodic environments on a
  strip.
\newblock {\em Ann. Probab.}, 47(4):2494--2528, 2019.

\bibitem{Dolgopyat_Goldsheid_21}
D.~Dolgopyat and I.~Goldsheid.
\newblock Constructive approach to limit theorems for recurrent diffusive
  random walks on a strip.
\newblock {\em Asymptot. Anal.}, 122(3-4):271--325, 2021.

\bibitem{Freedman_83}
D.~Freedman.
\newblock {\em Markov chains}.
\newblock Springer-Verlag, New York-Berlin, 1983.
\newblock Corrected reprint of the 1971 original.

\bibitem{Kaloshin_Sinai_00}
V.~Yu. Kaloshin and Ya.~G. Sinai.
\newblock Nonsymmetric simple random walks along orbits of ergodic
  automorphisms.
\newblock In {\em On {D}obrushin's way. {F}rom probability theory to
  statistical physics}, volume 198 of {\em Amer. Math. Soc. Transl. Ser. 2},
  pages 109--115. Amer. Math. Soc., Providence, RI, 2000.

\bibitem{Katok_Hasselblatt_95}
A.~Katok and B.~Hasselblatt.
\newblock {\em Introduction to the modern theory of dynamical systems},
  volume~54 of {\em Encyclopedia of Mathematics and its Applications}.
\newblock Cambridge University Press, Cambridge, 1995.
\newblock With a supplementary chapter by Katok and Leonardo Mendoza.

\bibitem{Kingman_Orey_64}
J.~F.~C. Kingman and S.~Orey.
\newblock Ratio limit theorems for {M}arkov chains.
\newblock {\em Proc. Amer. Math. Soc.}, 15:907--910, 1964.

\bibitem{Kozlov}
S.~M. Kozlov.
\newblock The averaging method and walks in inhomogeneous environments.
\newblock {\em Uspekhi Mat. Nauk}, 40(2(242)):61--120, 238, 1985.

\bibitem{Moshchevitin}
N.~G. Moshchevitin.
\newblock Khintchine’s singular {D}iophantine systems and their applications.
\newblock {\em Russian Math. Surveys}, 65(3):433--511, 2010.

\bibitem{Orey_61}
S.~Orey.
\newblock Strong ratio limit property.
\newblock {\em Bull. Amer. Math. Soc.}, 67:571--574, 1961.

\bibitem{Petersen}
K.~Petersen.
\newblock {\em Ergodic theory}, volume~2 of {\em Cambridge Studies in Advanced
  Mathematics}.
\newblock Cambridge University Press, Cambridge, 1989.
\newblock Corrected reprint of the 1983 original.

\bibitem{Sinai_99}
Ya.~G. Sinai.
\newblock Simple random walks on tori.
\newblock {\em J. Statist. Phys.}, 94(3-4):695--708, 1999.

\bibitem{Zeitouni_04}
O.~Zeitouni.
\newblock Random walks in random environment.
\newblock In {\em Lectures on probability theory and statistics}, volume 1837
  of {\em Lecture Notes in Math.}, pages 189--312. Springer, Berlin, 2004.

\bibitem{Zeitouni_06}
O.~Zeitouni.
\newblock Random walks in random environments.
\newblock {\em J. Phys. A}, 39(40):R433--R464, 2006.

\end{thebibliography}

\end{document}